\documentclass[reqno]{amsart}
\usepackage{csquotes}
\usepackage{amsfonts,amsmath,amsthm,amssymb}
\usepackage[normalem]{ulem}
\usepackage{enumerate,nicefrac}
\usepackage{hyperref,multicol,ulem}
\usepackage{rotating}
\usepackage{mathrsfs}
\usepackage[all]{xy}
\usepackage{tikz}
\usetikzlibrary{arrows, decorations.markings, shapes.geometric}

\tikzstyle{steps} = [rectangle, rounded corners, text width=3cm, minimum height=1.5cm,text centered, draw=black,ultra thick]
\tikzstyle{mid} = [draw=none,text centered]
\tikzstyle{endbox} = [ellipse, minimum width=3cm, text centered, draw=black,very thick]

\usepackage{mathtools}

\newtheorem{theorem}{Theorem}[section]
\newtheorem{lemma}[theorem]{Lemma}

\theoremstyle{definition}

\usepackage[ansinew]{inputenc}
\usepackage{ifthen}
\usepackage{animate}

\title[Additive Diophantine Equations]{Additive Diophantine Equations involving $S$-Units, Factorials and Ternary
Recurrences with repeated root}

\author{Vikas Godara}
\email{vikasgodara0529\symbol{64}gmail.com}
\author{Divyum Sharma}
\email{divyum.sharma\symbol{64}pilani.bits-pilani.ac.in}

\address{Department of Mathematics\\
Birla~Inst{i}tute~of~Technology~and~Science, Pilani 333\,031 \textsc{India}}

\begin{document}

\date{}
 \begin{abstract}
    Let $C_n=n2^n+1$ denote the $n$th Cullen number. There has been recent interest in finding all  Cullen numbers having a given Diophantine property.
    We prove that, for a fixed integer $k$ and bounded integers $a_1,\ldots,a_k$, the greatest prime divisor of 
    $C_n-a_1m_1!-\cdots-a_km_k!$ tends to infinity, in an effective way. We prove this for some more general families of ternary recurrence sequences as well. We also solve the Diophantine equation 
    \[
    C_n = m_1! + m_2! + s,
    \]
    where $s$ is a positive integer composed of primes $2,3,5,7$.
 \end{abstract}

\maketitle
%===============================
\section{Introduction}\label{sec_intro}

The only Fibonacci numbers that are also factorials are $1$ and $2$. This follows from Carmichael's primitive divisor theorem (see \cite{Rajagopal_Griffiths_Fibonacci_factorial} for a proof), and falls into the class of results that consider whether the intersection of the Fibonacci sequence with a specified infinite set is infinite. In 2002, Grossman and Luca \cite{GrossLucaSumfact} considered a more general setting and established that for a given non-degenerate binary recurrence sequence $\{u_n\}_{n\geq0}$ and fixed $k$, there are only finitely many (effectively computable) indices $n$ such that $u_n$ can be represented as a sum of $k$ factorials. In particular, they found all Fibonacci numbers expressible as the sum of two factorials. It turns out that $F_{12} = 5! + 4!$ is the largest solution. (Here, $F_n$ denotes the $n$th Fibonacci number, $F_1=F_2=1$, $F_{n+1}=F_n+F_{n-1}$ for $n\geq 2$.) In 2010, Bollman, Santos Hern{\'a}ndez and Luca \cite{BHL10} proved that $13 = 1! + 3! + 3!$ is the largest
Fibonacci number expressible as a sum of three factorials. Luca and Siksek \cite{Luca_Siksek_2010} treated the mirrored equation and determined all factorials that are the sum of at most three Fibonacci numbers.

In 2014, Sanchez and Luca \cite{SaLu14} considered the problem of representing members of a non-degenerate binary recurrence sequence $\{u_n\}_{n\geq0}$ as a linear combination of a factorial and an $S$-unit with bounded coefficients. For a finite set $\mathcal{P}=\{ p_1,p_2,\ldots,p_l=P\}$ of primes labelled increasingly, we let $\mathcal{S}_{\mathcal{P}}$ denote the set of all positive integers whose prime factors are in $\mathcal{P}$. Sanchez and Luca solved the equation
    \[
    F_n=\pm m!\pm s,\, s\in \mathcal{S}_{\mathcal{P}},
    \]
   % \textcolor{red}{Also, $P=\ \max\{p_1,p_2,\ldots,P\}$.} 
where $\mathcal{P}=\{2,3,5,7\}$ and found $F_{24} = 8! + 2^5\cdot3^3\cdot7$ to be the largest solution. 
In 2023, Luca and Noubissie \cite{LucaNoubissieCombinationfact} considered a variant of this problem with 
$\{u_n\}_{n\geq0}$ a non-degenerate ternary recurrence sequence whose characteristic polynomial has a repeated root. (This in turn implies that all characteristic roots are integers.) 
They showed that if $(n,m,s)$  is a \textit{non-degenerate} solution of the equation
    \[
     u_n=Am!+Bs,\,  \, s\in \mathcal{S}_{\mathcal{P}},
    \]
 where $\max\{|A|,|B|\}\leq K$ for a fixed $K$, then $n$ can be effectively bounded. As a particular case, we have the sequence $\{C_n\}_{n\geq0}$  of Cullen numbers, where $C_n=n2^n+1$. Indeed, it satisfies the recurrence relation $C_n = 5\,C_{n-1}-8\,C_{n-2} +4\,C_{n-3}$ and the characteristic polynomial
    \[
    x^3-5x^2+8x-4=(x-2)^2(x-1)
    \]
has a repeated root. There has been recent interest in finding all  Cullen numbers having a given Diophantine property. Luca and Noubissie   established that  if 
    \[
    C_n=\pm m!+s,\ m\geq 2,
    \]
 and $\mathcal{P}=\{2,3,5,7\}$, then $n\leq 8$. Indeed, $C_8=4!+3^4\cdot 5^2$. We refer to \cite{BerPiYougFibCul, BerrizbeitiaLuca2012, Ybilu2019, Cullen05, Lucastan, DiegoM2014, DiegoM2015, MeherRout2023} and the references therein for other results on Cullen numbers.

Recently, B\'erczes, Hajdu, Luca and Pink \cite{BHLP_23} considered the Diophantine equation
    \begin{align*}
         u_n=a_1m_1!+\cdots+a_km_k!+Bs,\, s\in \mathcal{S}_{\mathcal{P}}
    \end{align*}
    in integers $n,m_1,\ldots,m_k,s$, where $\{u_n\}_{n \geq 0}$ is a non-degenerate binary recurrence sequence. As an application, they solved the equation
    \begin{align*}
         F_n= m_1!+m_2!+s,\, s\in \mathcal{S}_{\mathcal{P}},
    \end{align*}
   where $\mathcal{P}=\{2,3,5,7\}$.

%===============================
In this paper, we consider a variant of this problem where $\{u_n\}$ is a ternary recurrence sequence $\{ u_n\}_{n \geq 0}$ satisfying
\[
    u_n = r_1u_{n-1} + r_2 u_{n-2} +r_3u_{n-3} \ \ \ (n\geq 3),
\]
where $r_1,r_2,r_3$ are non-zero integers and
the characteristic polynomial has a double root and $1$ as a root. Let 
    \[
    f(X) = X^3-r_1X^2-r_2X-r_3= (X-\alpha)^2(X-\beta)
    \]
 be the characteristic polynomial of $\{u_n\}_{n \geq 0}$. As already noted, it follows that $\alpha$ and $\beta$ are integers by virtue of the fact that $f$ has a double root (see \cite{LucaNoubissieCombinationfact}). Let $\gamma = \max{\{|\alpha|, |\beta|\}}$. Since $\alpha$ or $\beta$ equals $1$, it follows that $\gcd(r_1,r_2,r_3)=1$. We assume that $\alpha/\beta \neq \pm 1$. We can write 
    \begin{equation}\label{u_n}
    u_n = (an+c)\alpha^n+b\beta^n, \text {     where     } a,b,c\in \mathbb{Q}, \ \ \ a \neq 0.
\end{equation}  
Let
\begin{equation*}%\label{Y}
    Y=\max \left\{\left|r_{1}\right|,\left|r_{2}\right|,\left|r_{3}\right|,\left|u_{0}\right|,\left|u_{1}\right|,\left|u_{2}\right|, 11 \right\} .
\end{equation*}

Our first result is regarding the greatest prime divisor of the difference of $u_n$ and linear combinations of factorials.
 Let us denote the greatest prime dividing an integer $y$ by $P(y)$.
\begin{theorem}\label{th:general} 
    Let $k,A$ be positive integers and $\{u_n\}_{n\geq 0}$ be a non-degenerate ternary recurrence sequence with a double root and $1$ as a root. Assume that $|a_i|\leq A$ for $i=1,2, \ldots, k$ and the unknown integers $m_i$ satisfy $m_1>m_2>\cdots > m_k \geq 1 $.  Put  $c_1=Y^8$, $c_2= 2.02 \cdot 10^{12}  \log^2{Y}$ and let $n_1= n_1(k)$ be the largest %\textcolor{red}{ non-degenerate}
    integer solution of the inequality 
    \[
    n < 1.45\ (\log{4(A+1)} + 10.49  \ c_2 \log^3{n})^k.
    \]
    Then,
    \[
    P\bigg( u_n-\sum_{i=1}^k a_im_i!  \bigg) > c_3(n),
    \]
    whenever $\sum_{i=j_1}^{j_2}a_{i}m_{i}!\notin \{(an+c)\alpha^n,b\beta^n\}$ for $0< j_1\leq j_2\leq k$, and $n>c_4 := \max{\{c_1,n_1 \}}$, where
    \[
    c_3(n) := \bigg( \frac{n}{1.45} \bigg)^\frac{1}{2k+2} \bigg( \frac{\log{4A}}{4} + 2.63  \ c_2 \log^3{n} \bigg)^{-\frac{1}{2}}.
    \]
\end{theorem} 

 Our next result deals with 
Diophantine equations of the form
    \begin{equation}\label{eq1coro}
        u_n= a_1m_1!+a_2m_2!+ \cdots +a_km_k! + b_1s, %\qquad \max{\{|a_1|,|a_2|,\ldots |a_k|,b_1\}}\leq A 
    \end{equation}
   where  $s\in \mathcal{S_{\mathcal{P}}}$. 
We say that a solution $(n,m_1,\ldots,m_k,s)$ is \textit{non-degenerate} if $\sum_{i=j_1}^{j_2}a_{i}m_{i}!\notin \{(an+c)\alpha^n,b\beta^n\}$ for $0< j_1\leq j_2\leq k$.
\begin{theorem}\label{th:corollary}
    Let $k,A$ be positive integers and $\{u_n\}_{n\geq0}$ be a non-degenerate ternary recurrence sequence with double root and $1$ as a root. % Further, assume that $k \geq 1$ and $A \geq 1$ are fixed positive integers. 
    Consider the Diophantine equation \eqref{eq1coro}
    %\begin{equation*}
     %   u_n= a_1m_1!+a_2m_2!+ \cdots +a_km_k! + b_1s, %\qquad \max{\{|a_1|,|a_2|,\ldots |a_k|,b_1\}}\leq A 
    %\end{equation*}
    with $ \max{\{|a_1|,|a_2|,\ldots |a_k|,|b_1|\}}\leq A$, $m_1>m_2\cdots >m_k\geq 1$ and $s\in \mathcal{S_{\mathcal{P}}}$. Then for all non-degenerate solutions $(n,m_1,\ldots,m_k,s)$, we have 
    %in integer unknown satisfying $s\in \mathcal{S_{\mathcal{P}}}$ and $m_1>m_2\cdots m_k\geq 1$. Then 
    \begin{equation*}
        n \leq c_5 := \max{\{ c_4 , c_6\} },
    \end{equation*}
    where $c_4$ is defined in the statement of Theorem \ref{th:general} and 
    \[
    c_6  :=  2^{3k+3}c_7\log^{3k+3}(c_7(3k+3)^{3k+3})
    \]
    with
    \[
    c_7 := 1.45(\max{\{P,A\}})^{2k+2}(2\, c_2\log{4A})^{k+1}.
    \]
\end{theorem}   
Next, we completely solve an equation as in the above theorem.
%\begin{theorem}\label{th:specfic_p1}
 %   Let $\{C_n\}_{n\geq0}$ denote the sequence of Cullen numbers and $\mathcal{P}= \{2, 3, 5, 7 \}$. Then, all solutions of the equation
  %  \begin{equation}\label{m1m2s}
   % C_n = m_1! + m_2! + s \qquad \text{in } n, m_1, m_2,s \in \mathbb{N}, m_1 \geq m_2, s\in \mathcal{S}_{\mathcal{P}} 
    %\end{equation}
 %   are given by
%\[    
%\begin{aligned}
 % &[n, m_1, m_2, s] \in \{[1, 1, 1, 1], [2, 1, 1, 7], [2, 2, 1, 6],[2, 2, 2, 5], [2, 3, 1, 2], [2, 3, 2, 1],\\
 % &[3, 2, 2, 21], [3, 3, 1, 18], [4, 1, 1, 63], [4, 4, 1, 40], [4, 4, 3, 35], [5, 4, 2, 135], [5, 5, 1, 40],\\
  %&[5, 5, 3, 35], [6, 3, 1, 378], [6, 4, 1, 360], [7, 6, 2, 175], [8, 6, 3, 1323], [9, 6, 1, 3888]\}.
%\end{aligned}
%\]
%We also solve the case $m=1$ of \cite[Theorem 1.3]{LucaNoubissieCombinationfact}.\\
%\end{theorem}  
\begin{theorem}\label{th:specfic_p}
    Let $\{C_n\}_{n\geq0}$ denote the sequence of Cullen numbers and $\mathcal{P}= \{2, 3, 5, 7 \}$. Then, all non-degenerate solutions of the equation
    \begin{equation}\label{m1m2s}
    C_n = m_1! + m_2! + s \qquad \text{in } n, m_1, m_2,s \in \mathbb{N}, m_1 \geq m_2, s\in \mathcal{S}_{\mathcal{P}} 
    \end{equation}
    are given by
    \[    
    \begin{aligned}
     &[n, m_1, m_2, s] \in \{[2, 2, 2, 5], [3, 2, 2, 21], [4, 4, 3, 35], [5, 4, 2, 135], [5, 5, 3, 35], [7, 6, 2, 175], \\
     &[8, 6, 3, 1323]\}
    \end{aligned}
    \]
    and all degenerate solutions of eq. \eqref{m1m2s} for $n<10^{66}$ are given by 
    \[    
    \begin{aligned}
      &[n, m_1, m_2, s] \in \{[1, 1, 1, 1], [2, 1, 1, 7], [2, 2, 1, 6], [2, 3, 1, 2], [2, 3, 2, 1], [3, 3, 1, 18], \\
     &[4, 1, 1, 63], [4, 4, 1, 40], [5, 5, 1, 40], [6, 3, 1, 378], [6, 4, 1, 360], [9, 6, 1, 3888]\},
    \end{aligned}
    \]  
    except possibly if $m_2=1$, $m_1>10^4$ and $\sqrt{C_n}\leq m_1!$ simultaneously.
%We also solve the case $m=1$ of \cite[Theorem 1.3]{LucaNoubissieCombinationfact}.\\
\end{theorem}  
%\textcolor{red}{ For Cullen numbers, the degenerate solution whenever $m_2=1$ and \\ $C_2 = 3! + 2! + 1$.}
\textit{Remark.} Taking $m_1=2$ and $m_2=1$, eq. \eqref{m1m2s} may be rewritten as $n2^n-1=1+s$ i.e., $W_n=1!+s$, where $W_n=n2^n-1$ is the $n$th Woodall number. In \cite [Theorem 1.3]{LucaNoubissieCombinationfact}, the equation 
\[
W_n=m!+s,\ s\in \mathcal{S_P}\ \textrm{ for }\ \mathcal{P}= \{2, 3, 5, 7 \}
\]
was solved for $m\geq 2$. Theorem \ref{th:specfic_p} implies that the only solution of $W_n = 1! +s $ when $n<10^{66}$ is $W_2 = 1!+ 6$.\\

%==========================================
A key tool in our proofs is Yu's $p$-adic analogue for a lower bound for linear
forms in logarithms of algebraic numbers. As in \cite{LucaNoubissieCombinationfact} and \cite{BHLP_23}, an inductive argument that we shall employ often involves comparing lower bounds of $p$-adic valuation of factorials and with upper bounds for $p$-adic valuation of $u_n-t$, for some suitably chosen integer $t$. We note that for $t=0$, such an upper bound was already established in \cite[Lemma 3.4]{LucaNoubissieCombinationfact}. For $t\neq 0$, we prove the required estimate in Lemma \ref{un<logn}. Further, we establish that there are only finitely many (effectively computable) members of the sequence $\{u_n\}$ which can be represented as a (non-degenerate) linear combination of a fixed number of factorials (see Lemma \ref{n<f(n,t)}).
To deal with degenerate cases in Theorem \ref{th:specfic_p}, we use Lemma \ref{p-1sol} proved in Section \ref{sec_prelim} using a Hensel-lifting type argument. To deal with non-degenerate cases, we follow ideas from \cite{LucaNoubissieCombinationfact}. In Section \ref{sec_prelim}, we record some preliminary results and prove certain auxiliary results for later use. Sections 3--5 are dedicated to the detailed proofs of Theorems \ref{th:general},  \ref{th:corollary} and \ref{th:specfic_p}, respectively.
%===============================
\section{Preliminaries}\label{sec_prelim}
%introduce $\nu_p$\\
Let $p$ be a prime number and $a$ be an integer, then the $p$-adic valuation is defined as  
\begin{align*}
\nu_{p}(a) =
\begin{cases}
\max\{\,k\in \mathbb{N} \cup 0 : p^{k}\mid a\,\}, & \text{if } a\neq 0, \\
\infty, & \text{if } a = 0.
\end{cases}
\end{align*}

%basic inequality\\
The next three lemmas record some basic properties of $\nu_p$ and an estimate involving $p$-adic valuation of factorials.
\begin{lemma}\label{padicid}
    For a prime number $p$ and positive integers $a$ and $b$, we have
    \begin{enumerate}[(i)]

   %     \item   $\nu_p(-a)=\nu_p(a)$,

        \item   $\nu_p(ab)=\nu_p(a)+\nu_p(b)$,
        
        \item   $\nu_p(a+b)\geq \textit{min}(\nu_p(a),\nu_p(b))$. Here, equality holds 
         if $\nu_p(a)\neq \nu_p(b)$.
        %then  $\nu_p(a+b)=\textit{min}(\nu_p(a),\nu_p(b))$,

        % \item   If $\nu_p(a)= \nu_p(b)$  then  $\nu_p(a+b)\geq \textit{min}(\nu_p(a),\nu_p(b))$.
    
    \end{enumerate}
\end{lemma}
\begin{lemma}\label{vpfactorial}
    For a prime number $p$ and positive integer $a$
     \[
     \nu_p(a!) =    \sum_{i=1}^{\infty} \left\lfloor \frac{a}{p^{i}} \right\rfloor.
     \]
\end{lemma}
\begin{lemma}\label{v(a!)>a/2p}
    \cite[Lemma 1] {GrossLucaSumfact} Let $p$ be a prime number and let $a$ be a positive integer. If $a\geq p$ then $\nu_p(a!)>\frac{a}{2p}$.
\end{lemma}

%============================================
%Yu's theorem 
%============================================
Next, we recall the definition and some properties of the absolute logarithmic height function. 
The absolute logarithmic height $h(\eta)$ of an algebraic number $\eta $ with degree $d(\eta)$ over $\mathbb{Q}$ is defined by 
\begin{align*}
    h(\eta)= \frac{1}{d(\eta)}\bigg{(}\log{|a_0|} + \sum^{d(\eta)}_{i=1} \log{\max{\{|\eta^{(i)}|,1\}}}\bigg{)},
\end{align*}
where the minimal polynomial is given by 
\begin{align*}
    f(X)=a_0\prod^{d(\eta)}_{i=1}(X-\eta^{(i)})\in \mathbb{Z}[X].
\end{align*}
We will use the following properties of the absolute logarithm height function $h(\cdot)$: 
\begin{lemma}\label{height}
If $\delta_1,\delta_2$ are algebraic numbers, then
   % \begin{align*}\label{height_of_product}
   %     h(\delta_1\pm\delta_2)&\leq h(\delta_1)+h(\delta_2)+\log 2,\\
    %    h(\delta_1\delta_2^{\pm 1})\ \ &\leq h(\delta_1)+h(\delta_2),\\
     %   h(\delta_1^m) \quad &=|m|h(\delta_1) \ (m\in\mathbb{Z}).
    %\end{align*}
    \begin{enumerate}
        \item  $h(\delta_1\pm\delta_2)  \leq h(\delta_1)+h(\delta_2)+\log 2,$
        \item  $h(\delta_1\delta_2^{\pm 1})\ \leq h(\delta_1)+h(\delta_2),$
        \item  $h(\delta_1^m) =|m|h(\delta_1) \ (m\in\mathbb{Z}).$
    \end{enumerate}
\end{lemma}

Let $\mathbb{K}$ be an algebraic number field of degree $D$ over $\mathbb{Q} $ embedded in $\mathbb{C}$. If $\pi$ is a prime ideal in the ring $\mathcal{O}_{\mathbb{K}}$ of algebraic integers in $\mathbb{K}$, we denote by $e_{\pi}$ and $f_{\pi} $ the ramification index and the inertial degree of $\pi$, respectively. Let $p$ be the prime number above $\pi$ and $\nu_{\pi}(\eta)$ be the order at which $\pi$ appears in the prime factorization of the principal ideal $\eta \mathcal{O_{\mathbb{K}}}$.
%$D$ be the degree of number field $\mathbb{K}$ over $\mathbb{Q} $ embedded in $\mathbb{C}$. 
Suppose $\eta_1,\eta_2,\ldots,\eta_l \in \mathbb{K}\setminus\{0,1\}$ and $d_1,\ldots,d_l\in \mathbb{Z}$. Let $B^*= \max{\{|d_1|,\ldots,|d_l|,3\}}$ and $\Lambda= {\eta_1 }^{d_1}\cdots {\eta_l }^{d_l}-1$.
A key tool for our proofs is the following estimate of Yu for linear forms in logarithms.

\begin{lemma}\label{Yu}
    [Yu \cite{yu1999p}] Let $H_j\geq \max{ \{h(\eta_j), \log{p}\}}$, for $j=1,\ldots,l$. If $\Lambda\neq0$, then 
    \begin{align*}
        \nu_{\pi}(|\Lambda|) \leq 19(20\sqrt{l+1}D)^{2(l+1)}e_{\pi}^{l-1}\frac{p^{f_{\pi}}}{(f_{\pi}\log{p})^2}\log{(e^5lD)}\ H_1\cdots H_l \log{B^*}.
    \end{align*}
\end{lemma}
The following lemma is due to \cite{Petho1986}. (See \cite[Appendix B]{SmartPbook}  for a proof.)
\begin{lemma}\label{4.3}
    Let $u,v \geq 0, h\geq 1 $ and $x\in \mathbb{R}$ be the largest solution of $x= u+ v(\log{x})^h$. Then 
    \[
    x < \max{\{ 2^h(u^{\frac{1}{h}} + v^{\frac{1}{h}}\log{(h^hv))^h}, 2^h(u^\frac{1}{h} + 2e^2)^h  \}}.
    \]
\end{lemma}
We now record some results proved by Luca and Noubissie \cite {LucaNoubissieCombinationfact} regarding the sequence  $\{u_n\}$ for later use.% Put
%\begin{equation}\label{Y}
    %Y:=\max \left\{\left|r_{1}\right|,\left|r_{2}\right|,\left|r_{3}\right|,\left|u_{0}\right|,\left|u_{1}\right|,\left|u_{2}\right|, %11 \right\} .
%\end{equation}
\begin{lemma}\label{abc}
    \cite[Lemma 3.1] {LucaNoubissieCombinationfact} The numerator of $a,b,c$ in eq. \eqref{u_n} is at most $4Y^3$ and the denominator is at most $Y^3$. In particular, 
    \begin{equation*}
        \text{max}\{h(a),h(b),h(c)\}\leq log(4Y^3).
    \end{equation*}
\end{lemma}
\begin{lemma}\label{u_n=0}
    \cite[Lemma 3.2] {LucaNoubissieCombinationfact} If $u_n=0$, then $n < 39Y \log{Y}$.
\end{lemma}
\begin{lemma} \label{un>alpha}
    \cite[Lemma 3.3] {LucaNoubissieCombinationfact} Assume $n>Y^8$. If $|\beta|> |\alpha|$, then 
    \[
    |u_n|> \frac{|\beta|^{n}}{2Y^3}.
    \]
 Otherwise, if $|\alpha|>|\beta|$, then 
    \[
    |u_n|>\frac{n|\alpha|^n}{6Y^3}.
    \]
\end{lemma}
%Statement and proof of generalization of Lemma 5.2 \\
The next three lemmas will be used in the proof of Theorem \ref{th:specfic_p}.

\begin{lemma}\label{L:ineq_nm1}
    Let $n, m_1, m_2,s $ be positive integers as in \eqref{m1m2s} with $m_1\geq 6$. Then $n>m_1$.
\end{lemma}
\begin{proof}
    By eq. \eqref{m1m2s}, we obtain $n2^n > m_1!$ as $m_2!+s > 1$. Further, %using Stirling's formula,  
    \[
         n2^n>m_1!>m_12^{m_1},
    \]
   % Taking logarithms on both sides.
    %\[
   % 2n\log{n} > 2m_1 \log{m_1}, %> \frac{m_1 \log{m_1}}{2},
   % \]
    implying that
    \[
    n> m_1.
    \]
\end{proof}
The case $k=1$ of the following lemma was established in \cite[Lemma 5.2]{LucaNoubissieCombinationfact}.
\begin{lemma}\label{p-1sol}
  Let $k$ be a positive integer. For each fixed integer $t$ and odd prime $p$ there are exactly $p-1$ numbers $n$ in $\{0,1,\ldots,p^k(p-1)-1\}$ such that $p^k$ divides $n2^n-t$.
\end{lemma}
\begin{proof}
  % \textcolor{blue}{Suppose $\nu_p(t) \geq k$. Then $\nu_p(n2^n-t)\geq k $ if and only if $v_p(n)\geq k$. So, there are exactly $p-1$ numbers, namely $0,p^k,2p^k,\ldots,(p-2)p^k$, in $\{ 0,1,\ldots, p^{k}(p-1)-1\}$ such that $p^{k}$ divides ${n2^n -t}$. Therefore, we can now assume that $\nu_p(t) < k$.}

    We prove the result by induction on $k$. The lemma holds for $k=1$ by
     \cite[Lemma 5.2]{LucaNoubissieCombinationfact}. Let $j\geq 1$.
    Suppose that there are exactly $p-1$ numbers $n$ in $\{0, 1,\ldots, p^{j}(p-1)-1\}$ such that $p^j$ divides $n2^n-t$. We denote them as $a_{j1}, a_{j2}, \ldots, a_{j(p-1)}$.    So, there are exactly $p(p-1)$ numbers, namely  $a_{j1},a_{j1}+p^j(p-1),a_{j1}+2p^j(p-1),\ldots, a_{j1}+p^j(p-1)^2, \ldots, a_{j(p-1)}, a_{j(p-1)}+p^j(p-1),\ldots, a_{j(p-1)}+p^j(p-1)^2$ in $\{0, 1,\ldots, p^{j+1}(p-1)-1\}$ such that $p^j$ divides $n2^n-t$. %We denote them as $a_{j1},a_{j1}+p(p-1),a_{j1}+2p(p-1),\ldots, a_{j1}+p(p-1)^2, \ldots, a_{j(p-1)}, a_{j(p-1)}+p(p-1),\ldots, a_{j(p-1)}+p(p-1)^2$.
    
    Let $i$ denote $\nu_p(t)$. If $i \geq j$, then $\nu_p(a_{j1}2^{a_{j1}}-t)\geq j $ if and only if $v_p(a_{j1})\geq j$. So, there is exactly one number $n$ in  \{$a_{j1},a_{j1}+p^j(p-1),a_{j1}+2p^j(p-1),\ldots, a_{j1}+p^j(p-1)^2$\}\ such that $p^{j+1}$ divides ${n2^n-t}$. 
    
    %If $i=j$, then for $a_{j1}2^{a_{j1}}-t$ to be divisible by $p^j$, we must have $v_p(a_{j1})\geq j$. So,  there is exactly one number  $ n \in \{a_{j1},a_{j1} + (p-1)p^j, \ldots , a_{j1} + (p-1)^2p^j\}$ such that $p^{j+1}$ divides ${n2^n-t}$.  \\   
    Suppose $i < j$. Since $a_{j1}2^{a_{j1}}-t$ is divisible by $p^j$, we have $v_p(a_{j1}) = i$. Dividing $a_{j1}$ by $p^j$, we can write
    \begin{equation}\label{E:aj1_quotrem}
    a_{j1}=p^jb_{j1}+p^iy_{j1} 
    \end{equation}
    for some integers $b_{j1}, y_{j1}$ with $y_{j1} \equiv 1,\ldots ,\,p-1 \pmod{p}$. Let $v$ denote the least residue of $b_{j1}$ modulo $p$. (We suppress the dependence on $j$ for the sake of simplicity.) Then
    \begin{equation*}%\label{eq:defn_x}
        x:=\nu_p(b_{j1}-v) > 0.
    \end{equation*}
    Using \eqref{E:aj1_quotrem}, write
    \[
    a_{j1}2^{a_{j1}}-t= p^j(b_{j1}-v)2^{p^jb_{j1}+p^iy_{j1}} + (p^iy_{j1}+vp^j)2^{p^jb_{j1}+p^iy_{j1}} - t.
    \]
    Since $p^j$ divides $a_{j1}2^{a_{j1}}-t$, we have
    \[
       (p^iy_{j1}+vp^j)2^{p^jb_{j1}+p^iy_{j1}}-t \equiv 0 \pmod{p^j}.
       \]     
    Further, 
        \begin{equation}\label{E:a2_temp}
a_{j1}2^{a_{j1}}-t\equiv (p^iy_{j1}+vp^j)2^{p^jb_{j1}+p^iy_{j1}} - t  \pmod{p^{j+1}}
        \end{equation}
 as $x$ is positive. Thus, there is a unique integer $d \in \{0, 1,\ldots, p-1\}$ such that 
    \begin{align}\label{E:inverse_sub}
    \begin{split}
    (p^iy_{j1}+vp^j)2^{p^jb_{j1}+p^iy_{j1}}-t &\equiv dp^j \pmod{p^{j+1}}\\ 
    \textrm{i.e.,\ }  p^i2^{p^jb_{j1}+p^iy_{j1}} &\equiv (t+dp^j)(y_{j1}+vp^{j-i})^{-1} \pmod{p^{j+1}}.
  \end{split}
    \end{align}
    (Note that $\gcd(p,y_{j1}+vp^{j-i})=1$.) We claim that there is exactly one number $n \in \{a_{j1},a_{j1} + (p-1)p^j, \ldots , a_{j1} + (p-1)^2p^j\}$ such that $p^{j+1}$ divides ${n2^n-t}$.  Let $c \in \{0, 1,\ldots, p-1\}$. Then, as noted earlier, $p^j$ divides 
       \[
       (a_{j1}+cp^j(p-1))2^{a_{j1}+c(p-1)p^j}-t.
       \]
Using \eqref{E:a2_temp}, \eqref{E:inverse_sub} and the consequence of Euler's theorem that $2^{\varphi(p^{j+1})}\equiv 1\pmod{p^{j+1}}$, we find
        \begin{align*}
          &  (a_{j1}+c(p-1)p^j)2^{a_{j1}+c(p-1)p^j}-t\\
          % &=(p^jb_{j1}+ p^iy_{j1} +c(p-1)p^j) 2^{p^j b_{j1}+ p^iy_{j1}+c (p-1)p^j}-t\\
          % &=p^j(b_{j1}-v)2^{p^jb_{j1}+ p^iy_{j1}+ c(p-1)p^j} + (p^iy_{j1}+(v+c(p-1))p^j)2^{p^jb_{j1}+ p^iy_{j1} + c(p-1)p^j} - t\\
 %          &\equiv (p^iy_{j1}+(v+c(p-1))p^j)2^{p^jb_{j1}+ p^iy_{j1} + c(p-1)p^j} - t \pmod{p^{j+1}}.
 %        \end{align*}
 %    Now, using ,  
 % \begin{align*}
        & \equiv(p^iy_{j1} + (v + c(p-1)p^j))2^{p^jb_{j1}+p^iy_{j1}+c(p-1)p^j} -t \\
         & \equiv (p^iy_{j1}+vp^j)2^{p^jb_{j1}+p^iy_{j1}}2^{c(p-1)p^j} + c(p-1)p^j2^{p^jb_{j1}+p^iy_{j1}}2^{c(p-1)p^j} -t \\
          &\equiv  dp^j + c(p-1)p^{j-i}(y_{j1}+vp^{j-i})^{-1}(t+dp^j)   \pmod{p^{j+1}}.
    \end{align*}
%where Euler's theorem has been deployed to conclude that $2^{\varphi(p^{j+1})}\equiv 1\pmod{p^{j+1}}$. \\
     For fixed $d$, the congruence $c(p-1)p^{j-i}{(y_{j1}+vp^{j-i})}^{-1}(t+dp^j)  \equiv -dp^j  \pmod{p^{j+1}} $ holds for a unique $c\in \{0, 1,\ldots, p-1\}$ as $p^j$ is the highest power of $p$ dividing the coefficient of $c$. This proves the claim.
Hence, there is exactly one number $n \in \{a_{j1},a_{j1} + (p-1)p^j, \ldots , a_{j1} + (p-1)^2p^j\}$ such that $p^{j+1}$ divides ${n2^n-t}$.

    Similarly, there are exactly $p-2$ numbers $n \in \{a_{j2},a_{j2} + (p-1)p^j, \ldots , a_{j2} + (p-1)^2p^j,a_{j3}, \ldots, a_{j(p-1)}+(p-1)^2p^j\}$ such that $p^{j+1}$ divides ${n2^n-t}$.
    
    So, by the Principle of Mathematical Induction, there are exactly $p-1$ numbers $n $ in $\{ 0,1,\ldots, p^k(p-1)-1\}$ such that $p^k$ divides ${n2^n -t}$. 
 
\end{proof}
%============================================
%============================================
The following lemma summarises the procedure in the proof of \cite[Lemma 5.1] {LucaNoubissieCombinationfact}.
\begin{lemma} \label{njpj}
    Let $p$ be an odd prime and $N,t$ be integers with $N>0$. Let $n_0$ be one of the numbers in $\{0,1,\ldots,p(p-1)-1\}$ such that $p$ divides $n_02^{n_0}+1-t$ (cf. Lemma \ref{p-1sol}). For $j\geq 1$, define 
      \begin{align*}
        & n_{j-1} = n_0 + (p-1)pl_1 + \cdots (p-1)p^{j-1}l_{j-1}, \\
        & l_j \equiv 2^{-n_{j-1}} \bigg{(} \frac{n_{j-1}2^{n_{j-1}}+1-t}{p^j}  \bigg{)} \pmod{p}.
    \end{align*}
    Suppose $J$ is the smallest integer such that $n_{J}>N$. Then $\nu_p(C_n-t)< J$ for all $n\leq N$.
    %Set $k=\nu_p (C_{n_{J-1}}-t)$. Then $\nu_p(C_n-t)\leq k$ for all $n\leq N$.
\end{lemma}    
    % For a fixed prime $p$ and fixed integer $N$ such that $n<N$. Consider $n= n_0+p(p-1)l$for  $n_0$ as in [\cite{LucaNoubissieCombinationfact}, Lemma 5.2 ], where $l=l_1+l_2p+\cdots + l_kp^{k-1}+\cdots $, $l_i\in\{1,2,\ldots,p-1\}$ and $l_j,n_j$ are defined as  
  
    %  then 
    
    % the minimum value of $n$ such that $p^k$ divides $(C_n-t)$ can be determined, where $n =n_0 + p(p-1)l$ and 
    % \begin{align*}
    %     & l=l_1+l_2p+\cdots + l_kp^{k-1}+ \cdots, \qquad l_i \in \{1,2,\ldots,p-1\}, \\
    %     & n_{j-1} = n_0 + (p-1)pl_1 + \cdots (p-1)p^{j-1}l_{j-1}, \\
    %     & l_j \equiv 2^{-n_{j-1}} \bigg{(} \frac{n_{j-1}2^{n_{j-1}}+1-t}{p^j} \pmod{p} \bigg{)}.
    % \end{align*}
%============================================
%============================================

%====================================================
%\begin{lemma}\label{Luca}
 %   Let $p$ be a prime number. If $n>Y^8$, then 
  %     v_p(u_n)\leq 10^{12}\frac{p}{logp}(logp+logY)log^2n.
   % \end{equation*}
%\end{lemma}

As mentioned in Section \ref{sec_intro}, we prove an estimate for $\nu_p(u_n-t)$ in the following lemma.

\begin{lemma}\label{un<logn}
    Let $\{u_n\}_{n\geq 0}$ be a ternary recurrence sequence as in Theorem \ref{th:general}. Suppose $n>c_1$. Let $t\neq u_n$ be an integer. Further, suppose that $t \neq b$ if $\beta =1$ and $t \neq an+c$ if $\alpha =1$.
    Then for any prime $p$, we have 
\[
v_p(u_n-t) < 
 \begin{cases*}
     {} c_2 \, p \log^2{n} & \mbox{if t=0}, \\
{}  c_2 \, p \log^2{n} \log^+{t}  & \mbox{if $t\neq 0$},

\end{cases*}
\] 
where $\log^+{t}:=\max \{1,\log{|t|}\}$.
%where 
%\[
%c_2 := 1.33(10^{17})p \log^2{Y}. 
%\]
\end{lemma}
\begin{proof}The case $t=0$ was established in \cite[Lemma 3.4]{LucaNoubissieCombinationfact}. Now, suppose $t\neq0$. \\
    \textit{Case 1: Assume $\beta = 1$.} Using \eqref{u_n}, we get $u_n-t = (an+c)\alpha^n+b-t $. Let  $\Lambda = \frac{(an+c)\alpha^n}{(t-b)}-1$. (Note that $t\neq b$ and $\Lambda \neq 0$.) Then
    \[
        \nu_p(u_n-t) =  \nu_p(t-b) + \nu_p(\Lambda).
    \]
     To bound $\nu_p(\Lambda)$, we apply Lemma \ref{Yu} with the parameters  
    \begin{align*}
        &l=2,\quad \qquad\eta_1= (an+c)(t-b)^{-1}, \qquad \ \eta_2 =\alpha,  \\
       &d_1 = 1, \qquad \quad d_2=n, \qquad B^*=n   .
    \end{align*}
    Using Lemma \ref{height} and Lemma \ref{abc}, we have
    \begin{align*}
         h((an+c)(t-b)^{-1}) \leq &\,  h(an+c) +  h(t-b)   \\ 
                \leq & \ h(a) + h(n)+h(c)+\log{2} +  h(t) +h(b) + \log{2} \\ 
                \leq & \ 3 \, \log{(4Y^3)} + \log{n}+ 2\, \log{2} +\log^{+}{t}  \\
                \leq & \ 3\, \log{(4Y^3)}\log{n} \log^+{t} \\
  %      h(t-b) \leq &\ h(t) +h(b) + \log{2} \leq \log^{+}{t} + \log{(4Y^3)} + \log{2} \\
    %            \leq & 2 \, \log^+{t} \, \log{(4Y^3)}  , \\
    \end{align*}
    and
    \begin{align*}
        h(\alpha) = &\ \log{|\alpha|} \leq \log{Y}
    \end{align*}
    as $|\alpha| \leq |r_3| = |\alpha^2\beta| \leq Y$. 
    Hence by Lemma \ref{Yu}, we obtain
    \begin{align*}
        \nu_p(\Lambda) \leq  &\ 19(20\sqrt{3})^6\frac{p}{(\log{p})^2}\log{(2e^5)} \ H_1H_2 \log{n}   \\ 
         \leq & \ 5.61 \cdot 10^{11}\, p \log{Y} \log{(4Y^3)} \log^2{n} \log^+{t}  \\
         \leq & \ 2.01 \cdot 10^{12}\, p \log^2{Y} \log^2{n} \log^+{t}. \\
    \end{align*}
    This, in conjunction with the fact that
    \[
    \nu_p(t-b) \leq \frac{\log{(|b-t|)}}{\log{p}}  \leq \log{(4Y^3)} \log^+{t} \leq 3.6\log{Y} \log^+{t}
    \]
    yields
    \[
    \nu_p(u_n-t) < 2.02  \cdot 10^{12} \, p \log^2{Y} \log^2{n} \log^+{t}. 
    \]
    \textit{Case 2: Assume $\alpha = 1$.} Using \eqref{u_n}, we get $u_n-t = an+c+b\beta^n -t $.\\ Let  $\Lambda = \frac{b\beta^n}{(t-(an+c))}-1$. (Note that $t\neq an+c$ and $\Lambda \neq 0$.) Then
    \[
        \nu_p(u_n-t) =  \nu_p(t-(an+c)) + \nu_p(\Lambda).
    \]
     To bound $\nu_p(\Lambda)$, we apply Lemma \ref{Yu} with the parameters  
    \begin{align*}
        &l=2,\quad \qquad \eta_1 = b(t-(an+c))^{-1}, \qquad \ \eta_2 =\beta,  \\
       &d_1 = 1, \qquad \quad  d_2=n, \qquad \qquad \qquad\qquad B^*=n   .
    \end{align*}
    Using Lemma \ref{height} and Lemma \ref{abc}, we can write
    \begin{align*}
        %h(b) \leq &\  \log{(4Y^3)}    ,\\
        h(b(t-(an+c))^{-1}) \leq &\ h(b) +h(t) +h(an+c) + \log{2} \\
            \leq & \ 3\, \log{(4Y^3)} +  \log^{+}{t} + \log{(n)} + 2\, \log{2} \\
                \leq & \  3 \, \log^+{t}  \log{(4Y^3)} \log{(n)}   \\
    \end{align*}
    and
    \begin{align*}
        h(\beta) = &\ \log{\beta} \leq \log{Y}.
    \end{align*}
    As before, we obtain
    \begin{align*}
        \nu_p(\Lambda) \leq  &\ 19(20\sqrt{3})^6\frac{p}{(\log{p})^2}\log{(2e^5)} \ H_1H_2 \log{n}   \\ 
         \leq & \ 5.61 \cdot 10^{11}\, p\log{(4Y^3)} \log{Y} \log^2{n} \log^+{t}  \\
         \leq & \ 2.01 \cdot 10^{12}\, p \log^2{Y} \log^2{n} \log^+{t}. \\
    \end{align*}
    Since
    \[
    \nu_p(t-(an+c)) \leq \frac{\log{(|an+c-t|)}}{\log{p}}  \leq \log^2{(4Y^3)} \log{n} \log^+{t} \leq 16\,\log^2{Y} \log{n} \log^+{t},
    \]
    we get
    \[
    \nu_p(u_n-t) < 2.02  \cdot 10^{12} \, p \log^2{Y} \log^2{n} \log^+{t}. 
    \]
\end{proof}

In the next lemma, we prove that there are only finitely many (effectively computable) members of the sequence $\{u_n\}$ which can be represented as a (non-degenerate) linear combination of a fixed number of factorials.
\begin{lemma}\label{n<f(n,t)}
   Let $k$ and $A$ be fixed positive integers and  $\{u_n\}_{n\geq 0}$ be a ternary recurrence sequence as in Theorem \ref{th:general}. %In addition, let $k\geq1$ and $A\geq1$ be fixed positive integers.
   Consider the equation
    \begin{equation}\label{u=a_im_i}
         u_n = a_1m_1! + \ldots+ a_km_k! \quad \text{where} \quad |a_i| \leq A, (1\leq i\leq k)
    \end{equation}
   in integers $(n,m_1,\ldots,m_k)$ with $\sum_{i=j_1}^{j_2}a_{i}m_{i}!\notin \{(an+c)\alpha^n,b\beta^n\}$ for $0< j_1\leq j_2\leq k$ and 
    \begin{equation}\label{m_1>m_k>1}
        m_1>m_2>\cdots >m_k\geq 1.
    \end{equation} 
    Then, we have $n\leq \max{(c_1,n_0)}$, where $n_0:=n_0(k)$ is the largest positive integer solution of the inequality 
    \[
    n<\frac{(\log{4A}+10.49 \, c_2 \log^3{n})^k}{\log{\gamma}}.
    \]    
\end{lemma}
\begin{proof}
    The claim follows immediately if $n \leq c_1$. Thus, we assume that $n>c_1$. We may assume that there is no vanishing subsum on the right hand side of eq. \eqref{n<f(n,t)}, i.e. we have 
    \begin{equation}\label{sum0}
        \sum_{i\in I \subset \{1,2,\ldots,k\}}a_im_i!\neq 0
    \end{equation}
    for each non empty $I \subset \{1,2,\ldots,k\}$ because if not, we obtain an equation similar to eq. \eqref{sum0} with fewer terms.
    For $j=1,2,\ldots, k$ put 
    \[
    N_j = \sum_{i=1}^ja_{k+1-i}m_{k+1-i}!.
    \]
    We show by induction that 
    \begin{equation}\label{induction}
    \log{|4N_j|} < (\log{4A}+ 10.48  \, c_2\log^3{n})^j.
    \end{equation}
    For $j=1$, we have $|4N_j| = |4a_km_k!|$. Additionally, since $n>c_1$, using Lemma \ref{un<logn} with $t = 0$ and $p = 2$, we obtain 
    \begin{equation}\label{v2(un)<clogn}
      \nu_2(u_n)<  2 \,c_2\log^2{n}.  
    \end{equation}
    Furthermore, it is evident from eq. \eqref{m_1>m_k>1} that
    \[
    \nu_2(u_n) = \nu_2(a_1m_1!+a_2m_2!+\cdots + a_km_k!) \geq \nu_2(m_k!).
    \]
    If $m_k\geq4 \,(2 \,  c_2 \log^2{n})$ then by Lemma \ref{v(a!)>a/2p}, $\nu_2(m_k!) > 2 \, c_2 \log^2{n}$, contradicting eq.\eqref{v2(un)<clogn}. Therefore, $m_k<8\, c_2 \log^2{n}$ and 
    \begin{align}\label{log(4N1)< first}
    \begin{split}
    \log{|4N_1|} & = \log{|4a_km_k!|} \leq \log{|4a_k|}+m_k \log{m_k} \\
         &\leq  \log{4A} + 8 \, c_2 \log^2{n} \log{(8 \, c_2 \log^2{n})}.
  \end{split}
    \end{align}
    Since $n \leq 8 \, c_2<\frac{(\log{4A}+10.48 \, c_2\log^3{n})}{\log{\gamma}}$, we assume that $n>8 \, c_2$. Therefore, using eq. \eqref{log(4N1)< first}, we have 
    \[
    \log{|4N_1|} \leq \log{4A} + 8 \, c_2 \log^2{n} \log{(n \log^2{n})} = \log{4A} + 8 \, c_2 \log^2{n} \, ( \log{n} + \log{( \log^2{n})}).
    \]
    Since $n > c_1(\geq11^8)$, it follows that $\log{\log^2{n}} < 0.31\log{n}$. Thus, by eq. \eqref{log(4N1)< first}, we obtain
    \[
    \log{|4N_1|} < \log{4A} + 10.48 \, c_2 \log^3{n}.
    \]
    Now, assume that eq. \eqref{sum0} holds for some $1\leq j< k$. We reformulate eq. \eqref{u=a_im_i} as $u_n - N_j = a_{k-j}m_{k-j}!+ \cdots + a_1m_1!$. Using eq. \eqref{sum0} we conclude that $u_n\neq N_j$ and $N_j\neq 0$. Therefore, Lemma \ref{un<logn} can be applied with $p = 2$ and $t = N_j$. We obtain
    \begin{equation}\label{v2(un-Nj)}
        \nu_2(u_n-N_j)<2 \, c_2\log{|4N_j|} \log^2{n}.
    \end{equation}
    Additionally, by virtue of eq. \eqref{m_1>m_k>1}, we have
    \[
    \nu_2(u_n-N_j)= \nu_2(a_1m_1!+ \cdots + a_{k-j}m_{k-j}!) \geq \nu_2(m_{k-j}!)
    \]
    If $m_{k-j}\geq 8\, c_2 \log{|4N_j|} \log^2{n}$, then by Lemma \ref{v(a!)>a/2p} we conclude that 
    $\nu_2(m_{k-j}!) > 2 \, c_2 \log{|4N_j|\log^2{n}}$ which contradicts eq. \eqref{v2(un-Nj)}. Hence, $m_{k-j} < 8 \, c_2 \log{|4n_j|} \log^2{n}$. Therefore, we have 
    \begin{equation}\label{log(4akmk)< logn}
        \log{|4a_{k-j}m_{k-j}!|} < \log{4A} + 8 \, c_2 \log{|4N_j|} \log^2{n} \log{(8 \, c_2 \log{|4N_j|} \log^2{n})}.
    \end{equation}
    We may assume that $n> 8 \, c_2 \log{|4N_j|}$. In fact, by eq. \eqref{induction} we obtain that $n<8 \, c_2 (\log{4A}+10.48 \, c_2 \log^3{n})^j$ if $n\leq 8\,  c_2 \log{|4N_j|}$. As $j\leq k-1$, we get  
    \[
    n< 8 \, c_2 (\log{4A} + 10.48\, c_2 \log^3{n})^{k-1}
    \]
    which implies that $n\leq n_0$. As before, we may write by $n> 8 \, c_2 \log{|4N_j|}$ and eq. \eqref{log(4akmk)< logn} that 
    \begin{align*}
        \log{|4a_{k-j}m_{k-j}|} & < \log{4A} +8 \, c_2\log^2{n} \log(n\log^2{n}) \log{|4N_j|} \\ 
                                & < \log{4A} +10.48 \, c_2\log^3{n} \log{|4N_j|}.
    \end{align*}
    It is clear that 
    \[
    |4N_{j+1}| \leq |4N_j| + |4a_{k-j}m_{k-j}!|.
    \]
    Thus, 
    \begin{align*}
        |4N_{j+1}| & < |4N_j| + \exp\{\log{4A} +10.48\,  c_2\log^3{n} \log{|4N_j|}\} \\
                    & = |4N_j| +4A |4N_j|^{10.48 \, c_2\log^3{n}}.
    \end{align*}
    This leads to the inequality
    \begin{equation}\label{log(4Nj+1)< first}
    \log{|4N_{j+1}|} < \log{4A} + 10.48 \log^3{n} \log{|4N_j|} + \log{\bigg(1+\frac{1}{4A|4N_j|^{10.48 \, c_2\log^3{n}-1}}\bigg)}.
    \end{equation}
    Also, $\log{\bigg(1+\frac{1}{4A|4N_j|^{10.48 \, c_2\log^3{n}-1}}\bigg)}<0.1$ because $10.48\, c_2\log^3{n}-1\geq 1$, $A\geq1$ and $|N_j|\geq1$. Thus, from eq. \eqref{log(4Nj+1)< first} we have
    \begin{equation} \label{log(4Nj)< second}
        \log{|4N_{j+1}|} < \log{4A} + 10.48 \, c_2 \log^3{n}\log{|4N_j|} +0.1.
    \end{equation}
    The combination of eqs. \eqref{induction} and \eqref{log(4Nj)< second} gives 
    \begin{align*}  %\label{log(4Nj+1)< second}
        \log{|4N_{j+1}|} & <  \log{4A} + 10.48 \, c_2 \log^3{n} (\log{4A}+10.48 \, c_2 \log^3{n} )^j +0.1 \\
           & < (\log{4A}+10.48 \, c_2 \log^3{n} )^{j+1}
    \end{align*}
    %Since, 
   % \[
    %\log{4A}  + 10.84 \, c_2 \log^3{n} (\log{4A}+10.48\, c_2 \log^3{n} )^j < (\log{4A}+10.48 \, c_2 \log^3{n} )^{j+1}-1.
    %\]
   % Using eq. \eqref{log(4Nj+1)< second}, we may conclude that 
   % \[
    %\log{|4N_{j+1}|} < (\log{4A}+10.48 \, c_2 \log^3{n} )^{j+1}
    %\]
    which completes the induction. Note that $u_n= N_{k}$\\
    Using Lemma \ref{un>alpha}, we have 
   % \[
    %|u_n|   > \frac{n|\alpha|^n}{6Y^3}.
    %\]
    %Therefore,
    \[
    \frac{\gamma^n}{ Y^3}  < |2\, u_n|,
    \]
    yielding
    \[
    n  < \frac{\log{|4u_n| + 3\, \log{Y}}}{\log{\gamma}}   = \frac{\log{|4N_k|} + 3\, \log{Y} }{\log{\gamma}}.
    \]
    %\begin{align*}
     %   |u_n|  & > \frac{n|\alpha|^n}{6Y^3}, \\
      %  |\alpha|^n & < |u_n| < |4u_n|, \\
       % n & < \frac{\log{|4u_n|}}{\log{|\alpha|}} = \frac{\log{|4N_k|}}{\log{|\alpha|}}.
    %\end{align*}
    Hence
    \[
    n < \frac{1}{\log{\gamma}}(\log{4A}+ 10.49\, c_2 \log^3{n})^k.
    \]
    Thus $n \leq n_0$.
\end{proof}
%==================================================================================================================
\section{Proof of Theorem \ref{th:general}}
\begin{proof}
    Assume that $n> \ c_4=\max{\{c_1,n_1\}}$. Since $n_1\geq n_0$, Lemma \ref{n<f(n,t)} implies that 
    \[
        u_n-(a_1m_1!+ \cdots + a_km_k!) \neq 0.
    \]
    Recall that $|a_i|\leq A$ and 
    \begin{equation}\label{m1>mk>1,th1}
    m_1>m_2> \cdots > m_k \geq 1.
    \end{equation}
    Thus, we may write 
    \begin{equation}\label{un-aimi!=0}
    u_n= a_1m_1!+ \cdots + a_km_k!+s,
    \end{equation}
    where $s\neq 0$ is some integer. 
    %Employing an inductive argument similar to that applied in Lemma \ref{n<f(n,t)}, we obtain an explicit upper bound for $n$ in terms of $P,A,k$ and $Y$. in eq. \eqref{un-aimi!=0}. This leads to an explicit lower bound for $P$ and therefore for $P(u_n-(a_1m_1!+a_2m_2!+ \cdots + a_km_k!))$.   
    Assume that eq. \eqref{un-aimi!=0} does not have a vanishing subsum on the right hand side i.e., 
    \[
    \sum_{i\in I \subset \{1,2,\ldots,k\}}a_im_i!+ \delta s \neq 0,
    \]
    for each non empty subset $I \subset \{1,2,\ldots,k\}$ and each $\delta \in \{0,1\}$. Indeed, if there is an index set $I \subset \{1,2,\ldots,k\}$ and $\delta \in \{0,1\}$ such that 
    \[
    \sum_{i\in I \subset \{1,2,\ldots,k\}}a_im_i!+ \delta s = 0,
    \]
    then from eq. \eqref{un-aimi!=0}, we obtain that
    \[ 
    u_n = 
    \begin{cases*}
    {}  \sum\limits_{i\in I \subset \{1,2,\ldots,k \} \setminus I} a_im_i!+s,     & \mbox{if $\delta =0$}, \\
    {}  \sum\limits_{i\in I \subset \{1,2,\ldots,k\}\setminus I}a_im_i!,  & \mbox{if $ \delta \neq 0.$}
    \end{cases*}
    \]
    For $\delta=1$, we have an equation similar to eq. \eqref{u=a_im_i} which for $n>\max(c_1,n_0)$ is not possible and for  $\delta=0$ we have an equation similar to eq. \eqref{un-aimi!=0} with fewer terms.

    Suppose $|s|=m_i!$ for some $i= 1,2,\ldots,k$, then we have 
    \begin{equation*}
        u_n = a_1m_1! + \ldots+ (a_i+1)m_i!+ \cdots +  a_km_k!.
    \end{equation*}
    Using Lemma \ref{n<f(n,t)} and the preceding eq. we conclude that $n \leq c_4$, which is a contradiction to $n>c_4$. 

    Let $m_{k+1}=0$ and $m_0$ be such that $m_0!>\max\{|s|,m_1!\}$. There exists an integer $i_0$ with $0\leq i_0\leq k$ such that 
    \begin{equation*}
        m_{k+1-i_0}! < |s|< m_{k-i_0}!.
    \end{equation*}
    Also, for $i= 1,2,\ldots, k+1$, let
    \begin{equation*}%\label{define ti}
        t_i = 
        \begin{cases}
            a_{k+1-i}m_{k+1-i}!, & \text{if } i < i_0+1, \\
            s, & \text{if } i=i_0+1, \\
            a_{(k+1)-(i-1)}m_{(k+1)-(i-1)}!, & \text{if } i > i_0+1. 
        \end{cases}
    \end{equation*}
    For $j=1,2, \ldots,k+1$, put
    \begin{equation*}
        N_j = \sum_{i=1}^jt_i.
    \end{equation*}
    Using induction, we now show that
    \begin{equation}\label{PMIassumptionTh1}
        \log{|4N_j|}<(\log{4A} + 2.62 \, c_2 P^2\log^3{n})^j,
    \end{equation}
    where $P= \max\{p:p|s\}$. \\
    For $j=1$, it is easy to observe that
    \begin{equation*}
        N_1 = t_1 = 
        \begin{cases}
            s & \text{if } |s|<m_k!, \\
            a_{k}m_{k}! & \text{if } |s|>m_k!. 
        \end{cases}
    \end{equation*}
    \textit{Case 1: Assume $|s| < m_k!$.}

    %We let $P= \max{\{p : p|s\}}$.
    Since $n>c_4>39\, Y \log{Y}$, using Lemma \ref{u_n=0}, we have $u_n \neq 0$. Using Lemma \ref{un<logn} with $t=0$ and every prime factor $p|s$, we conclude that
    \begin{equation}\label{v(un)<c2P,th1}
        \nu_p(u_n) < c_2\, P \log^2{n}.
    \end{equation}
    Furthermore, it is evident from eq. \eqref{m1>mk>1,th1} that
    \begin{equation}\label{v2(un)>v2(mk),th1}
     \nu_p(a_1m_1!+a_2m_2!+\cdots + a_km_k!) \geq \nu_p(m_k!).
    \end{equation}
    If $m_k \geq 2\, c_2 {P}^2 \log^2{n} $, then by Lemma \ref{v(a!)>a/2p} with $p \leq {P}$ we obtain 
    \[
    \nu_p(m_k!)> \frac{m_k}{2p}\geq c_2\, {P} \log^2{n}.
    \]
    Hence, by eqs. \eqref{un-aimi!=0}, \eqref{v(un)<c2P,th1} and \eqref{v2(un)>v2(mk),th1} yields
    \begin{equation}\label{vp(un)=vp(s),th1}
        \nu_p(u_n)=\nu_p(s) \qquad \text{for} \ p|s.
    \end{equation}
    Therefore, we conclude from eqs. \eqref{v(un)<c2P,th1} and  \eqref{vp(un)=vp(s),th1} that 
    \[
    \log{|s|}= \sum_{p|s}\nu_p(s)\log{p} = \sum_{p|s}\nu_p(u_n)\log{p} < c_2\, {P} \log^2{n} \log{P} \, \pi{(P)}.
    \]
    As $\pi{(P)} < \frac{2P}{\log{P}} $ (see Corollary $1$ in \cite{RosserSchoenfeld1962}), the preceding inequality leads to 
    \begin{equation*}%\label{case1eq1}
        \log{(4|N_1|)} = \log{(4|s|)} < \log{4} +2\, c_2 {P}^2 \log^2{n}. 
    \end{equation*}
    Here eq. \eqref{PMIassumptionTh1} holds for $j\geq1$.
    If $m_k< 2 \, c_2 {P}^2 \log^2{n}$, then 
    \begin{equation*}
        \log{|4m_k!|} < \log{4} + (2\, c_2 {P}^2\log^2{n})\log{(2\, c_2 {P}^2\log^2{n})},
    \end{equation*}
    because $m_k!\leq m_k^{m_k}$. We may assume that $n\geq 2c_2{P}^2$, since otherwise we obtain $P>n^{\frac{1}{2}}(2\, c_2)^{-\frac{1}{2}}$, which is better than the stated inequality. Therefore, we get 
    \begin{equation*}
    \log{|4m_k!|} < \log{4} + (2\, c_2 {P}^2\log^2{n})\log{(n\log^2{n})}.
    \end{equation*}
    Since $n > c_1(\geq11^8)$, it follows that $\log{\log^2{n}} < 0.31\log{n}$. Thus, by previous eq., we can state that
    \[
    \log{|4m_k!|} < \log{4} + 2.62 \, c_2 {P}^2 \log^3{n}.
    \]
    Using $\log{|4N_1|}= \log{|4s|}$, we conclude that 
    \begin{equation*}%\label{case1eq2}
        \log{|4N_1|} < \log{4} + 2.62 \, c_2 {P}^2 \log^3{n}.
    \end{equation*}
    Thus eq. \eqref{PMIassumptionTh1} holds for $j=1$ in this case.\\
    \textit{Case 2: Assume $|s| > m_k!$.} Then we have $N_1=a_km_k!$. Using the same justification as in Case $1$, we establish that $\log{(4|s|)}<\log{4} + 2 \, c_2{P}^2\log^2{n}$, for $m_k \geq 2\, c_2 \, {P}^2 \log^2{n}$. This together with $m_k!<|s|$ and $|a_k|\leq A$ implies that 
    \begin{equation*}%\label{case2eq1}
        \log{|4N_1|} = \log{|4a_km_k!|} < \log{|4a_ks|} < \log{4A} + 2 \, c_2 {P}^2 \log^2{n}.
    \end{equation*}
    Suppose $m_k< 2\, c_2{P}^2\log^2{n}$. Using the same justification as in Case $1$, we conclude 
    \begin{equation*}%\label{case2eq2}
        \log{|4N_1|} = \log{|4a_km_k!|}  < \log{4A} + 2.62 \, c_2 {P}^2 \log^3{n}.
    \end{equation*}
    Hence, eq. \eqref{PMIassumptionTh1} holds for $j=1$ in this case as well, 
    proving the induction hypothesis for $j=1$.\\
    
    Assume that eq. \eqref{PMIassumptionTh1} holds for $j=J$ where  $1\leq J<k+1$. 
    Now, we prove that eq. \eqref{PMIassumptionTh1} is true for $j=J+1$. We can rewrite eq. \eqref{un-aimi!=0} as 
    \begin{equation}\label{un-nj=aimi!+deltas,Th1}
         u_n-N_J = a_1m_1!+ \cdots + a_lm_l! +\delta s 
    \end{equation}
    where $\delta \in \{0,1\}$ and $ l= l(\delta,J,k) := k+1-J-\delta$. Clearly, we can see that $N_{J+1} = N_J +t_{J+1}$, where
    \begin{equation}\label{tj=alml!,Th1}
        t_{J+1} = 
        \begin{cases}
           
            a_{l}m_{l}!, & \text{if } (\delta =0) \text{ or  } (\delta =1 \text{ and } |s|>m_l!),\\
             s, & \text{if } \delta =1 \text{ and } |s|<m_l!.
        \end{cases}
    \end{equation}
    Also, $N_J \neq 0$ and $u_n - N_J\neq 0 $. For every prime $p$, we obtain $\nu_p(u_n-N_J)< c_2\, p\log^+{N_J}\log^2{n} <c_2\, p\log{|4N_J|}\log^2{n}  $ by applying Lemma \ref{un<logn} with $t=N_J$ (Note that the conditions of Lemma \ref{un<logn} are satisfied as the solutions are non-degenerate.) Further, if $p\, |\, s$ then $p\leq {P}$. Hence,
    \begin{equation}\label{vp(un-Nj),TH1}
        \nu_p(u_n-N_J)< c_2\, {P} \log{|4N_J|}\log^2{n}.
    \end{equation}
    Using eq. \eqref{m1>mk>1,th1}, we conclude that
    \begin{equation}\label{>vp(ml!)}
        \nu_p(a_1m_1!+\cdots+ a_lm_l!)\geq \nu_p(m_l!).
    \end{equation}
    In order to determine $\log{|4t_{J+1}|}$, we divide the proof into three cases based on the value of $t_{J+1}$.

    \textit{Case I: Assume $\delta=1$ and $|s|<m_l!$.}\\
    By eq. \eqref{tj=alml!,Th1}, we get $t_{J+1}= s$. Suppose $m_l\geq  2 \, c_2 {P}^2\log{|4N_J|\log^2{n}}$. Let $p$ be a prime dividing $s$. Then Lemma \ref{v(a!)>a/2p} implies that
    \begin{equation*}
        \nu_p(m_l!)> \frac{m_l}{2p} \geq c_2\, {P} \log{|4N_J|}\log^2{n}.
    \end{equation*}
    Using eqs. \eqref{un-nj=aimi!+deltas,Th1}, \eqref{vp(un-Nj),TH1} and \eqref{>vp(ml!)}, we get that
    \begin{equation*}
        \nu_p(u_n-N_J)= \nu_p(s).
    \end{equation*}
    Therefore, %using eq. \eqref{vp(un-Nj),TH1} and $p\leq P$ we get that
    \begin{align*}
        \log{|s|}  & = \sum_{p|s}\nu_p(s)\, \log{p} = \sum_{p|s}\nu_p(u_n-N_J)\log{p} \\
                   & < c_2\, P\log{|4N_J|} \log^2{n}\, \pi({P})\, \log{P}.
    \end{align*}
    Since $\pi({P})<\frac{2P}{\log{P}}$, we conclude that
    \begin{equation}\label{case3eq1}
        \log{|4t_{J+1}|} = \log{|4s|} < \log{4} + 2\,c_2\, {P}^2\log{|4N_J|} \log^2{n}\,.
    \end{equation}
    Now, suppose $m_l<2\, c_2\, {P}^2\log|4N_J|\log^2{n}$. Using the same justification as in the case of $j=1$, we get that   
    \begin{equation*}
        \log{(4m_l!)}< \log{4}+ 2\, c_2 {P}^2 \log{|4N_J|}\log^2{n}\, \log{(2\, c_2 {P}^2  \log{|4N_J|}\log^2{n})}.
    \end{equation*}    
    If $n<2\, c_2 {P}^2 \log{|4N_J|}$ then by using eq. \eqref{PMIassumptionTh1} and $J\leq k$, we obtain
    \begin{equation}\label{case3eq2}
        n<2\, c_2 {P}^2 (\log{4A} + 2.62\, c_2 {P}^2 \log^3{n})^J< (\log{4A} + 2.62\, c_2 {P}^2 \log^3{n})^{J+1},
    \end{equation}
    which implies the theorem. Now assume  $n \geq 2\, c_2 {P}^2 \log{|4N_J|} $. So, 
    \begin{align*}
        \log{(4m_l!)} & < \log{4}+ 2\, c_2 {P}^2 \log{|4N_J|}\log^2{n}\, \log{(n\log^2{n})} \\
                      & <  \log{4}+ 2\, c_2 {P}^2 \log{|4N_J|}\log^2{n}\,(\log{n} + \log{(\log^2{n})}). 
    \end{align*}
    Since $n\geq 2\, c_2 {P}^2 \log{|4N_J|}$, it follows that $\log(\log^2{n})<0.31\, \log{n}$. Thus, by the previous equation, we  can state that
    \begin{equation*}
        \log{|4t_{J+1}|} = \log{|4s|} < \log{4m_l!} < \log{4} + 2.62\, c_2 \, {P}^2 \log{|4N_J|}\,  \log^3{n}.
    \end{equation*}
    \textit{Case II: Assume $\delta=1$ and $|s|> m_l!$.}\\
    By an argument similar to Case I, we obtain that
    \begin{equation}\label{case3eq3}
        \log{|4t_{J+1}|} = \log{|4s|}  < \log{4} + 2\, c_2 \, {P}^2 \log{|4N_J|}\,  \log^3{n}.
    \end{equation}
    \textit{Case III: Assume $\delta=0$.}\\
    By eq. \eqref{tj=alml!,Th1}, we have $t_{J+1}= a_lm_l!$. Let $p$ be a prime dividing $s$. We now use Lemma \ref{un<logn} with $t=N_J$. (Note that the conditions of Lemma \ref{un<logn} are satisfied as the solutions are non-degenerate.) We thus obtain
    \begin{equation}\label{case3eq4}
        \nu_p(u_n-N_J)< c_2\, {P} \log{|4N_J|}\, \log^2{n}.
    \end{equation}
    Using eqs. \eqref{m1>mk>1,th1} and \eqref{un-nj=aimi!+deltas,Th1}, we conclude that
    \begin{equation*}
        \nu_p(u_n-N_J) \geq \nu_p(m_l!).
    \end{equation*}
    If $m_l \geq  2\, c_2 \, {P}^2 \log{|4N_J|}\log^2{n}$, then by using Lemma \ref{v(a!)>a/2p} and $p\leq {P}$ we get 
    \begin{equation*}
        \nu_p(m_l!)> c_2 \, {P}^2 \log{|4N_J|} \log^2{n},
    \end{equation*}
    which contradicts the previous equation.\\ We may assume  $m_l <  2\, c_2 \, {P}^2 \log{|4N_J|}\log^2{n}$. Using the same justification as  in the preceding case, we may conclude that  
    \begin{equation*}
        \log{|4t_{J+1}|} =  \log{4a_lm_l!} < \log{4A} + 2.62\, c_2 \, {P}^2 \log{|4N_J|}\,  \log^3{n}.
    \end{equation*}
    Finally, eqs. \eqref{case3eq1}, \eqref{case3eq2}, \eqref{case3eq3} and \eqref{case3eq4} allow us to conclude that
    \begin{equation*}
        \log{|4t_{J+1}|}  < \log{4A} + 2.62\, c_2 \, {P}^2 \log{|4N_J|}\,  \log^3{n}.
    \end{equation*}
    Since $N_{J+1}=N_J + t_{J+1}$, we obtain
    \begin{equation*}
        |4N_{J+1}|< |4N_J| + \exp{\{\log{4A} + 2.62\, c_2 \, {P}^2 \log{|4N_J|}\,  \log^3{n}\}},
    \end{equation*}
    implying that
    \begin{equation*}
        |4N_{J+1}|< |4N_J| + 4A\, |4N_J|^{2.62\, c_2 {P}^2\log^3{n}}.
    \end{equation*}
    The previous inequality leads to 
    \begin{align*}
        \log{|4N_{J+1}|}&< \log{4A} + 2.62 \, c_2 {P}^2 \log{|4N_J|} \log^3{n} \\
                        &+ \log{\bigg(1+\frac{1}{4A\, {|4N_J|}^{2.62 \,  {P}^2 \log^3{n} - 1}}\bigg)}.
    \end{align*}
    Since $A\geq1,\ |N_J|\geq1, \ {P}\geq 2, \ n> n_1 $ and $c_2 \geq 1.33 \cdot 10^{17} \log^2{11}$, we get
    \begin{equation*}
         \log{|4N_{J+1}|} < \log{4A} + 2.62 \, c_2 {P}^2 \log^3{n} (\log{4A} + 2.62 \, c_2 {P}^2 \log^3{n})^J  +0.1.
    \end{equation*}
    The combination of eq. \eqref{PMIassumptionTh1} and the previous equation gives 
    \begin{align*}
        \log{|4N_{J+1}|} < (\log{4A} + 2.62 \, c_2 {P}^2 \log^3{n})^{J+1},
    \end{align*}
    finishing the induction. Note that $u_n = N_{k+1} $.
    Using Lemma \ref{un>alpha}, we get
    \[
    \frac{\gamma^n}{Y^3}  < |2u_n|,
    \]
    yielding
    \begin{align*}
    n  & < \frac{\log{|4u_n|} + 3 \, \log{Y} }{\log{\gamma}} = \frac{\log{|4N_{k+1}|}+ 3 \, \log{Y} }{\log{\gamma}} \\
        & < \frac{1}{\log{\gamma}}(\log{4A} + 2.63 \, c_2 {P}^2 \log^3{n})^{k+1}  \leq \frac{1}{\log{2}}\, {P}^{2k+2} \bigg{(}\frac{\log{4A}}{4} + 2.63 \, c_2 \log^3{n} \bigg{)}^{k+1}.
    \end{align*}
    Hence,
    \[
    {P} > \bigg( \frac{n}{1.45} \bigg)^\frac{1}{2k+2} \bigg( \frac{\log{4A}}{4} + 2.63  \ c_2 \log^3{n} \bigg)^{-\frac{1}{2}}.
    \]
\end{proof}
%==================================================================================================================
\section{Proof of Theorem \ref{th:corollary}}
\begin{proof}
    Assume $n>c_4$. We can rewrite eq. \eqref{un-aimi!=0} as 
    \begin{equation*}
        u_n-\sum_{i=1}^ka_im_i! = b_1s.
    \end{equation*}
    Using Theorem \ref{th:general} and the previous equation, we conclude that
    \begin{equation*}
       c_3(n) < P\bigg( u_n-\sum_{i=1}^ka_im_i \bigg) = P(b_1s) \leq \max{\{{P},A\}},
    \end{equation*}
    leading to
    \begin{equation}\label{n<<,Th2}
    n < 1.45\, (\max{\{{P},A\}})^{2k+2} \bigg( \frac{\log{4A}}{4} + 2.63 \, c_2 \log^3{n}  \bigg)^{k+1}.
    \end{equation}
    Since $A\geq 1$ and  $n \geq c_4$, we get 
    \[
    \frac{1}{4\times 2.63\, \log^3n} + \frac{1}{\log(4A)} < 0.73,
    \]
    which together with eq. \eqref{n<<,Th2} gives
    \[
    n < c_7 \log^{3k+3}{n}.
    \]
    Applying Lemma \ref{4.3} with 
    \[
    x=n, u=0, v= c_7, h = 3k+3,
    \]
    we get 
    \begin{align*}
        n  & < \max{\{ 2^{3k+3}c_7\log^{3k+3}(c_7(3k+3)^{3k+3}),2^{3k+3}(2e^2)^{3k+3}\}} \\
           & = 2^{3k+3}c_7\log^{3k+3}(c_7(3k+3)^{3k+3})=c_6.
    \end{align*}
     So, $n<\max{\{c_4,c_6\}}$.

\end{proof}
%==================================================================================================================
\section{Proof of Theorem \ref{th:specfic_p}}
\begin{proof}
    Applying Theorem \ref{th:corollary} with 
    \[
    k=2, \qquad \mathcal{P}= \{2,3,5,7\},\qquad A = 1,
    \]
    we conclude that for every solution of eq. \eqref{m1m2s}, we must have $n\leq10^{66}$. We divide the proof into two cases:\\
    
\noindent    \textit{Case I: Assume $m_1!\geq \sqrt{C_n}$} i.e.,
\begin{equation}\label{cn<m1}
        C_n\leq (m_1!)^2.
\end{equation}
We further divide Case I into subcases: \\

\noindent\textit{Case I(i): Assume $m_1\leq 10^4$.}\\
    By eq. \eqref{cn<m1}, we obtain that $n<236899$.
    For all such $n$, we have
        \begin{align*}
            &  \qquad  \nu_3(n2^n+1) \leq 12, \qquad \nu_5(n2^n+1) \leq 7, \qquad \nu_7(n2^n+1) \leq 6. 
        \end{align*}
\textit{Case I(i)(a): Assume $m_2 \geq 49$.} Eq. \eqref{m1m2s} implies $n>200$. Then by Lemma \ref{padicid}, we have 
    \[
    \nu_2(m_1!+m_2!)\geq \nu_2(m_2!) \geq \nu_2(49!)\geq 47 > 0=\nu_2(C_n).
    \]
    Similarly,
    \begin{align*}
        %& \nu_2(m_1!+m_2!)\geq \nu_2(m_2!) \geq 47 > \nu_2(C_n), \\
        & \nu_3(m_1!+m_2!)\geq \nu_3(m_2!) \geq 22 > \nu_3(C_n), \\
        & \nu_5(m_1!+m_2!)\geq \nu_5(m_2!) \geq 12 > \nu_5(C_n), \\
        & \nu_7(m_1!+m_2!)\geq \nu_7(m_2!) \geq 8 > \nu_7(C_n).
    \end{align*}
By eq. \eqref{m1m2s}, we obtain 
    \[
    \nu_p(s) = \nu_p(C_n)\ \textrm{ for }\ p=2,3,5,7.
    \]
But $s\in \mathcal{S}_{\mathcal{P}}$. Hence 
    \begin{equation}\label{E:s_fact}
        s= 2^{\nu_2(C_n)}3^{\nu_3(C_n)}5^{\nu_5(C_n)}7^{\nu_7(C_n)}.
    \end{equation}
Since $m_1\geq m_2\geq 49$, we have
    \[
    \nu_{11}(m_1!+m_2!)\geq4\ \textrm{ and }\ \nu_{13}(m_1!+m_2!)\geq3.
    \]
But $m_1!+m_2!=C_n-s$ by virtue of eq. \eqref{m1m2s}. We compute $\nu_{11}(C_n-s)$ for $200<n<236899$ and $s$ as in \eqref{E:s_fact}. It turns out that 
$\nu_{11}(C_n-s)\geq4$ only at $n \in \{36483,73205,131769,146410,159395,161051,  186397,203265,219615,222723, \\ 234256\}$.
Among these values, the inequality $\nu_{13}(C_n-s) \geq 3 $ never holds. Thus, we conclude that eq. \eqref{m1m2s} has no solution in this case.\\

\noindent \textit{Case I(i)(b): Assume $1 \leq m_2 \leq 48$ and $m_1 \geq 56$.} 
%For $m_2=1$, we have 
%$\nu_2(C_n-m_2!) = \nu_2 (n2^n)\geq n $ and $\nu_2(m_1!)< \nu(n2^n)<\nu_2((m_1!)^2)$.   
%\begin{align*}
 %   \text{For } m_1\leq10^4,\qquad v_2((m_1!)^2) \leq 19990 \qquad n\leq 19990
%\end{align*}
%By equation (\ref{m1m2s}) $v_p(C_n-m_2!)=\nu_p(m_1!-s)$.
First suppose $m_2\geq 2 $. Using  a short computer program, we obtain 
\begin{align*}
    & \nu_2(C_n-m_2!)=0, \qquad \nu_3(C_n-m_2!)\leq13, \\
    & \nu_5(C_n-m_2!)\leq 9, \qquad \nu_7(C_n-m_2!)\leq7.
\end{align*}
By Lemma \ref{vpfactorial} we have, 
\begin{align*}
    & \nu_2(m_1!) \geq 53, \qquad \nu_3(m_1!) \geq 26, \\
    & \nu_5(m_1!) \geq 13, \qquad \nu_7(m_1!) \geq 9.
\end{align*}
These, together with eq. \eqref{m1m2s} and Lemma \ref{padicid}, can be used to deduce that $v_p(C_n-m_2!)=\nu_p(m_1!+s) = \nu_p(s)$. Thus, $s=3^{v_3(C_n-m_2!)}5^{v_5(C_n-m_2!)}7^{v_7(C_n-m_2!)} $.

Since $\nu_3(m_1!)\geq 26$, we use the procedure described in the proof of Lemma \ref{p-1sol} to find possible values of $n$ such that $C_n-m_2!-s$ is divisible by $3^{26}$. (Here, we consider $C_n-m_2!-s\pmod{3^{30}}$). It turns out that $\nu_3(C_n-m_2!-s)\neq \nu_3(m_1!)$. So, eq. \eqref{m1m2s} has no solutions in this case.\\

Next, suppose $m_2=1$. Then \eqref{m1m2s} becomes $n2^n=m_1!+s$. Since $\nu_2(m_1!)<m_1<n\leq \nu_2(n2^n)$, we deduce that $\nu_2(m_1!)=v_2(s)=\nu_2(n2^n-s)$. Since $m_1\leq 10^4$, we get $\nu_2(n2^n-s)\nu_2(10^4!)=9995$. Now, let $p\in\{3,5,7\}$. For all $n<236899$, we have $\nu_p(n)<\nu_p(56!)\leq \nu_p(m_1!)$. Hence $\nu_p(n)=\nu_p(n2^n)=\nu_p(m_1!+s)=\nu_p(s)$. Therefore, $\nu_3(s)\leq 11$, $\nu_5(s)\leq7$ and $\nu_7(s)\leq 6$. As before, we use the procedure described in the proof of Lemma \ref{p-1sol} to find possible values of $n$ such that $n2^n-s$ is divisible by $3^{26}$. (Here, we consider $n2^n-s\pmod{3^{30}}$). We find that eq. \eqref{m1m2s} has no solution in this case either.\\

\noindent \textit{Case I(i)(c): Assume that $1\leq m_2\leq 48$ and $m_1 < 56$.} Then we have $n\leq 476.$ All such solutions in this case are given by 
\[    
\begin{aligned}
  &[n, m_1, m_2, s] \in \{[2, 3, 1, 2], [2, 3, 2, 1], [3, 2, 2, 21], [3, 3, 1, 18], [4, 4, 1, 40], [4, 4, 3, 35], \\
  &[5, 4, 2, 135], [5, 5, 1, 40], [5, 5, 3, 35], [6, 4, 1, 360], [7, 6, 2, 175], [8, 6, 3, 1323], \\
  &[9, 6, 1, 3888]\}.
\end{aligned}
\]
\noindent \textit{ Case I(ii): Assume $m_1> 10^4$.} Then we have $236899 \leq n \leq 10^{66}.$\\
\textit{Case I(ii)(a): Assume $m_2\geq 500$.} Then by Lemma \ref{padicid},  
$$ \nu_p(m_1!+m_2!)\geq \nu_p(m_2!)\geq \nu_p(500!)$$ for $p\in \{3,5,7\}$. Further, using Lemma \ref{vpfactorial}, $$\nu_3(m_1!+m_2!)\geq 247,\ \nu_5(m_1!+m_2!)\geq 124,\ \nu_7(m_1!+m_2!)\geq 82.$$ We apply Lemma \ref{njpj} with $p=3$, $N=10^{66}$ and $t=0$. Here, $p(p-1)=6$ and there are two residue classes modulo $6$ such that $3$ divides $n2^n+1$, namely $n_0 \in \{1,2\}$. For $n_0=1$, we get
%\begin{align*}
 %   & \nu_3(m_1!+m_2!)=\nu_3(m_2!)\geq \nu_3(500!)=247, \\
  %  & \nu_5(m_1!+m_2!)=\nu_5(m_2!)\geq \nu_5(500!)=124, \\
   % & \nu_7(m_1!+m_2!)=\nu_7(m_2!)\geq \nu_7(500!)=82.
%\end{align*}
\begin{align*}
     & n_{138}=2757614145106930270081057081158539402776859635842902126805823275421 ; \end{align*}
     and for $n_0=2$, we get
\begin{align*}
    & n_{138}=3748965004946665018258752266935970257963103092086460066359587819606.
\end{align*}
In both cases $n_{138}\geq N>n$, so $\nu_3(n2^n+1)<138$. Therefore, $\nu_3(s)<138$. In a similar manner, we find upper bounds for $\nu_p(s)$, when $p=5,7$. For $p=5$, we have $n_0 =3,4,6,17$. In these cases, we obtain 
\begin{align*}
    & n_{93} = 1244650605196477470301580667824245061531559793720522502019265072203; \\
    & n_{93} = 1795694848152108430374603592113726096902379193508535956140707676264; \\
    & n_{93} =  1358767469241923119082399935940451457976880852577230089606670816066;\\
    & n_{93} = 1924318815520452781692680587531291126323690582766162635381273157717, 
\end{align*}
respectively. Since $n_{93}\geq N>n$, we get $\nu_5(n2^n+1)<93$. Therefore, $\nu_5(s)<93$. When $p=7$, we have the cases $n_0 =5,6,10,26,27,31$ yielding 
\begin{align*}
    & n_{78} =23376667116957912273395168878053596583934978592913658754638298386469; \\
    & n_{78} = 26944746689754581236007271009151875823474002652201195796068635289134;\\
    & n_{78} = 24069582378334816208567848014057127858216459565384781083488608965992; \\
    & n_{78} = 6004003289610317916795511974189307812131311913908480006270103623040; \\
    & n_{78} = 9572082862406986879407614105287587051670335973196017047700440525705; \\
    & n_{78} =  6696918550987221851968191110192839086412792886379602335120414202563,
\end{align*}
respectively.  Therefore, $\nu_7(s)<78$. Now we calculate 
\[
    \max\{\nu_2(3^a5^b7^c-1) : 0 \leq a\leq 137,0\leq b \leq 92, 0 \leq c \leq 77\}. 
\]
We exclude the case where $a=b=c=0$ (since substituting $a=b=c=0 $ leads to $ n2^n=m_1!+m_2!$, but $\nu_2(m_1!+m_2!)\leq m_1<n\leq\nu_2(n2^n)$, a contradiction). We find that $\max\{\nu_2(3^a5^b7^c-1)\}\leq 20$. Since 
    \begin{align*}
    \nu_2(m_2!)&\geq \nu_2(500!) =494\\
    &>\max\{\nu_2(3^a5^b7^c-1): 0 \leq a\leq 137,0\leq b \leq 92, 0 \leq c \leq 77\},
    \end{align*}
we get 
    \[
   n\leq \nu_2(n2^n)=\nu_2(m_1!+m_2!+s-1)=\min(\nu_2(m_1!+m_2!),\nu_2(s-1))\leq 20<m_1,
    \]
contradicting Lemma \ref{L:ineq_nm1}. Hence, there is no solution in this case.\\ 

\noindent    \textit{Case(I)(ii)(b) Assume $2\leq m_2 < 500$.}  Then $s$ is odd.
By Lemma \ref{vpfactorial},
    \[
    \nu_3(m_1!)\geq 4996,\ \nu_5(m_1!)\geq 2499,\ \nu_7(m_1!)\geq 1665.
    \]
We now apply Lemma \ref{njpj} with $p=3$, $N=10^{66}$ and $t=m_2!$. In all cases, we find that $n_{140}>N>n$, so $\nu_3(C_n-m_2!)<140$. Likewise, for $p=5$, we obtain that $n_{110}>N>n$ in all cases implying that $\nu_5(C_n-m_2!)<110$ and, for $p=7$, we get $\nu_7(C_n-m_2!)<90$. We can rewrite eq. \eqref{m1m2s} as $C_n-m_2!=m_1!+s$. Since $\nu_p(C_n-m_2!)<\nu_p(m_1!)$ for $p\in \{3,5,7\}$, it follows from Lemma \ref{padicid} that $\nu_p(C_n-m_2!)=\nu_p(m_1!+s)=\nu_p(s)$ . Thus 
\[
s=3^a5^b7^c, \qquad 0\leq a\leq140,\ 0\leq b\leq110,\ 0\leq c\leq90.
\]
We can also rewrite eq. \eqref{m1m2s} as $n2^n+1-s=m_1!+m_2!$. For $s$ as above, we check computationally that $\nu_2(1-s)<n$. Hence, by Lemma \ref{padicid}, $\nu_2(n2^n+(1-s))=\nu_2(1-s)$. Further, $\nu_2(m_1!+m_2!)=\nu_2(m_2!)$. Therefore, $\nu_2(1-s)=\nu_2(m_2!)$.\\
Now we calculate 
\[
    \max\{\nu_2(3^a5^b7^c-1) : 0 \leq a\leq 140,0\leq b \leq 110, 0 \leq c \leq 90\}. 
\]
As in Case \textit{I(ii)(a)}, we can exclude the case where $a=b=c=0$. We find that $\max\{\nu_2(3^a5^b7^c-1) : 0 \leq a\leq 140,0\leq b \leq 110, 0 \leq c \leq 90\}\leq 20$. If $m_2>24$, then $\nu_2(m_2!)> \nu_2(24!) =22>\nu_2(1-s)$, which is a contradiction. Thus $m_2 \leq 24$.\\
Next, we rewrite eq. \eqref{m1m2s}) as $n2^n-m_1!=m_2!+(s-1)$. By computation, we find that $\nu_2(m_2!+s-1)\leq 25 $ but, by Lemma \ref{padicid}, $\nu_2(n2^n-m_1!)=\nu_2(m_1!)>\nu_2(10^4!)=9995$. So, there is no solution in this case.\\

\noindent    \textit{Case II: Assume $m_1!< \sqrt{C_n}$}, i.e. 
\begin{equation}\label{cn>m1}
        C_n> (m_1!)^2.
\end{equation}
We divide Case II into subcases: \\

\noindent     \textit{Case II(i): Assume $m_1\geq 560$.} \\

    % \textit{Case II(i)(a): Assume $m_2\geq 500$.}
\noindent    \textit{Case II(i)(a): Assume $m_2\neq 1$.} We proceed as in Case $I(ii)(a)$ if $m_2\geq 500$ and as in Case $I(ii)(b)$ if $1<m_2< 500$ and find that there are no solutions in this case.\\

% By an identical justification as for Case I(ii)(a), we conclude that there is no solution in this case.\\
%     \textit{Case II(i)(b): Assume $1<m_2 < 500$.}
% By an identical justification as for Case(I)(ii)(b), we conclude that there is no solution in this case.\\
\noindent    \textit{Case II(i)(b): Assume $m_2=1$.}
We can rewrite eq. \eqref{m1m2s} as $n2^n-m_1!=s$. By eq. \eqref{cn>m1}, we obtain $m_1!<s$. It follows from Lemmas \ref{vpfactorial} and \ref{L:ineq_nm1} that
\begin{align*}
    & \nu_2(m_1!) < m_1 <  n,  \qquad \qquad \quad \nu_3(m_1!) \geq \nu_3(560!)= 276, \\
    & \nu_5(m_1!) \geq \nu_5(560!) = 138, \qquad \nu_7(m_1!) \geq \nu_7(560!) = 92.
\end{align*}
Also we have, 
\begin{align*}
    & \nu_2(n2^n) \geq n,   \qquad \qquad \qquad \nu_3(n2^n) = \nu_3(n)\leq 140, \\
    & \nu_5(n2^n) = \nu_5(n) \leq 110,  \qquad \nu_7(n2^n) = \nu_7(n) \leq  80.
\end{align*}
So, by Lemma \ref{padicid}, we obtain $s=2^{\nu_2(m_1!)}3^{\nu_3(n)}5^{\nu_5(n)}7^{v_7(n)}<m_1!$, which is in contradiction to $m_1!<s$. So, we conclude that there is no solution in this case.\\

\noindent    \textit{Case II(ii): Assume $m_1 < 560$.}\\

\noindent    \textit{Case II(ii)(a): Assume $25\leq m_2<560$ and $m_1<560$.}

In this case $s$ is odd. We apply Lemma \ref{njpj} with $p=3$, $N= 10^{66}$ and $t=m_1!+m_2!$. In all cases, $n_{150}>N>n$. So, $\nu_3(n2^n+1-m_1!-m_2!)<150$. In a similar manner for $p=5,7$, we have $n_{110}>N>n$ and $n_{90}>N>n$, respectively. Therefore, $\nu_5(n2^n+1-m_1!-m_2!)<110$ and $\nu_7(n2^n+1-m_1!-m_2!)<90$. %in all cases $n_{110}>N>n$. So $\nu_5(n2^n+1-m_1!-m_2!)<110$. For $p=7$, in all cases $n_{90}>N>n$, so $\nu_7(n2^n+1-m_1!-m_2!)<90$. \\
We can rewrite eq. \eqref{m1m2s} as $C_n-m_1!-m_2!=s$. Thus, 
\[
s=3^a5^b7^c \  \  0\leq a\leq150,\ 0\leq b\leq110,\ 0\leq c\leq90.
\]
We rewrite eq. \eqref{m1m2s} as $n2^n+1-s=m_1!+m_2!$. It follows from Lemma \ref{padicid} that $\nu_2(n2^n+(1-s))=\nu_2(1-s)$ and $\nu_2(m_1!+m_2!)=\nu_2(m_2!)$ or $\nu_2(m_2!)+1$. So, $\nu_2(1-s)=\nu_2(m_2!)$ or $\nu_2(m_2!)+1$. By computation, we find that $\nu_2(s-1)\leq 20$. Thus, $m_2 \leq 24$. So, eq. \eqref{m1m2s} has no solution in this case.\\

\noindent    \textit{Case II(ii)(b): Assume $1\leq m_2<25$ and $244<m_1\leq 560$.}

We rewrite eq. \eqref{m1m2s} as $n2^n+1-m_1!-m_2!=s$. If $m_2\geq 2$ then $\nu_2(s)=\nu_2(n2^n+1-m_1!-m_2!)=0$. If $m_2=1$, then $\nu_2(s)=\nu_2(n2^n+1-m_1!-m_2!)=\nu_2(n2^n-m_1!)$. By Lemma \ref{L:ineq_nm1}, for $m_1 > 17$ we have $\nu_2(n2^n)>\nu_2(m_1!)$. So,
\[
\nu_2(s)=\nu_2(n2^n-m_1!)=\nu_2(m_1!)< 560.
\]
We rewrite eq. \eqref{m1m2s} as $n2^n+1-m_1!-m_2!=s$. For $p\in\{3,5,7\}$, $v_p(s)= \nu_p(n2^n+1-m_1!-m_2!)\leq \nu_p(m_1!)$. So, 
\[
\nu_3(s)\leq \nu_3(m_1!)<280 \qquad \nu_5(s)\leq \nu_5(m_1!)<140 \qquad \nu_7(s)\leq \nu_7(m_1!)<95.
\]
Hence, $s<2^{560}3^{280}5^{140}7^{95}$. By eq. \eqref{m1m2s}, 
\[
C_n - \sqrt{C_n}< C_n - m_1! = m_2!+s<25!+2^{560}3^{280}5^{140}7^{95}. 
\]
We therefore obtain $\sqrt{C_n}<25!+2^{560}3^{280}5^{140}7^{95}$. So $n\leq 3180$ and $m_1\leq 244$. Thus, we conclude that eq. \eqref{m1m2s} does not have a solution in this case.\\

\noindent    \textit{Case II(ii)(c): Assume $1\leq m_2<25$ and $122<m_1\leq 244$.}

By an argument similar to \textit{Case II(ii)(b)}, we obtain that
\begin{align*}
    & \nu_3(s)\leq \nu_3(m_1!)<121, \qquad \nu_5(s)\leq \nu_5(m_1!)<58, \\
    & \nu_7(s)\leq \nu_7(m_1!)<38,  \qquad \nu_2(s) < 244.
\end{align*}
Therefore, we have $n\leq 1343 $ and $m_1\leq 122$. Thus, we conclude that eq. \eqref{m1m2s} has no solution in this case.\\

\noindent    \textit{Case II(ii)(d): Assume $1\leq m_2<25$ and $69<m_1\leq 122$.}

By an argument similar to Case \textit{II(ii)(b)}, we find that
\begin{align*}
    & \nu_3(s)\leq \nu_3(m_1!)<58, \qquad \nu_5(s)\leq \nu_5(m_1!)<28, \\
    & \nu_7(s)\leq \nu_7(m_1!)<19,  \qquad \nu_2(s) < 122.
\end{align*}
Therefore, we have $n\leq 655 $ and $m_1\leq 69$. Therefore, there is no solution to eq. \eqref{m1m2s} in this case either.\\

\noindent    \textit{Case II(ii)(e): Assume $1\leq m_2<25$ and $1\leq m_1\leq 69$, $(m_1,m_2)\neq(2,1),(4,1)$}.
    
We now apply Lemma \ref{njpj} with $p=3$, $ N=10^{66}$ and $t=m_1!+m_2!$. In all cases, we find that $n_{150}>N>n$. So, $\nu_3(C_n-m_1!-m_2!)<150$. Similarly, for $p=5$, we obtain $n_{110}>N>n$ in all cases, which implies that $\nu_5(C_n-m_1!-m_2!)<110$ and, for $p=7$, we get $\nu_7(C_n-m_1!-m_2!)<90$. Also, $v_2(C_n-m_1!-m_2!)\leq\nu_2(69!)=66$. So, $$C_n\leq 69!+25!+2^{66}3^{150}5^{110}7^{90},$$ which implies that $n\leq 802$. All solutions in this case are given by
\[    
\begin{aligned}
  &[n, m_1, m_2, s] \in \{[1, 1, 1, 1], [2, 1, 1, 7], [2, 2, 2, 5], [3, 2, 2, 21], [4, 1, 1, 63], [6, 3, 1, 378]\}.
\end{aligned}
\]

\noindent    \textit{Case II(ii)(f): Assume  $(m_1,m_2)=(2,1),(4,1)$.}

We use the procedure described in the proof of Lemma \ref{p-1sol} to find possible values of $n$ such that (the non-zero integer) $C_n-m_1!-m_2!$ is divisible by $3^{140}$, $5^{99}$ and $7^{80}$. (Here, we consider $C_n-m_1!-m_2!\pmod{3^{200}},\pmod{5^{200}}, \pmod{7^{200}}$). It turns out that 
\begin{align*}
    & \nu_2(C_n-m_1!-m_2!)=\nu_2(s)\leq 2, \qquad \qquad \nu_3(C_n-m_1!-m_2!)=\nu_3(s)\leq 140, \\
    & \nu_5(C_n-m_1!-m_2!)=\nu_5(s)\leq 99, \qquad \quad \nu_7(C_n-m_1!-m_2!)=\nu_7(s)\leq 80.
\end{align*}
So, $$C_n\leq 4!+1!+2^{2}3^{140}5^{99}7^{80},$$ which implies that $n\leq 669$. Hence the only solution in this case is $[n, m_1, m_2, s]=[2, 2, 1, 6]$. This completes the proof of Theorem \ref{th:specfic_p}.

\end{proof}
%===============
%In \cite [Theorem 1.3]{LucaNoubissieCombinationfact}, for the case $n 2^n -1 = 1! +s $ with $n< 10^{58}$, if we apply Lemma \ref{njpj} than for $n_0 =0$ we obtain $n_{j-1}=1$ and $l_j = 0$ for all $p \in P$. So, we use the procedure described in the proof of Lemma \ref{p-1sol} to find possible values of $n<10^{58}$ such that (the non-zero integer) $n2^n-1-1!$ is divisible by $3^{125}$, $5^{88}$ and $7^{78}$. (Here, we consider $n2^n-2\pmod{3^{200}},\pmod{5^{200}}, \pmod{7^{200}}$). It turns out that 
%\begin{align*}
%    & \nu_2(n2^n-2)=\nu_2(s)\leq 1, \qquad \qquad \nu_3(n2^n-2)=\nu_3(s)\leq 125, \\
  %  & \nu_5(n2^n-2)=\nu_5(s)\leq 88, \qquad \quad \nu_7(n2^n-2)=\nu_7(s)\leq 78.
%\end{align*}
%So, $n2^n-1\leq 1!+2^{1}3^{125}5^{88}7^{78},$ which implies that $n\leq 613$. Hence the only solution of $n 2^n -1 = 1! +s $  is $2 \cdot 2^n-1 = 1!+ 6$.  

%============================
   \bibliographystyle{plain}
  \bibliography{ref}
\end{document}